\documentclass[12 pt]{article}
\usepackage[english]{babel}
\usepackage[utf8]{inputenc}
\usepackage{amsmath}
\usepackage{color,latexsym,amsthm,amsmath,amssymb,path,enumitem,url,bbm,subcaption,geometry,bbm}

\usepackage[pdftex]{graphicx}

\usepackage[utf8]{inputenc}
\newtheorem{theorem}{Theorem}[section]

\newtheorem{lemma}[theorem]{Lemma}

\newcommand{\parag}[1]{\vspace{2mm}

\noindent{\bf #1} }

\newcommand{\RR}{\mathbb R}

\newcommand{\pg}{{\mathsf{pg}}}
\newcommand{\pgset}{{\cal G}}
\newcommand{\pts}{\mathcal{P}}
\def\hv{{\hat{v}}}
\newcommand{\eps}{\varepsilon}
\newcommand{\pot}{{\mathsf{pt}}}

\title{Expected degrees in random plane graphs\thanks{Research work on this project was done as part of the 2024 NYC Discrete Math REU, funded by NSF award DMS-2349366, and by PSC/CUNY award 67511-00 55.	}}

\author{Neely Lovvorn\thanks{ University of North Alabama, AL, USA.
{\sl nlovvorn@una.edu}.}
\and
Oscar Murillo-Espinoza\thanks{California State University, Monterey Bay, CA, USA.
{\sl omurilloespinoza@csumb.edu}.}
\and 
Adam Sheffer\thanks{CUNY: Baruch College, NY, USA.
{\sl adamsh@gmail.com}.}
}

\date{}

\begin{document}
\maketitle
\begin{abstract} 
We prove that, for every set of $n$ points $\pts$ in $\RR^2$, a random plane graph drawn on $\pts$ is expected to contain less than $n/10.18$ isolated vertices. 
In the other direction, we construct a point set where the expected number of isolated vertices in a random plane graph is about $n/23.32$.
For $i\ge 1$, we prove that the expected number of vertices of degree $i$ is always less than $n/\sqrt{\pi i}$

Our analysis is based on cross-graph charging schemes. 
That is, we move charge between vertices from different plane graphs of the same point set. This leads to information about the expected behavior of a random plane graph.
\end{abstract}

\section{Introduction}

Random graph theory is an active and highly successful research field, with a deep theory and a wide variety of applications (for example, see \cite{Bollobas98,FK15}).
Unfortunately, the main techniques of this field fail when considering certain types of geometric graphs. 
In this work, we consider random \emph{plane graphs}.
While plane graphs are useful for modeling many problems, hardly anything is known about random plane graphs.

Let $\pts$ be a finite point set in $\RR^2$.
In a \emph{plane graph} of $\pts$, the vertices are the points of $\pts$ and the edges are non-intersecting line segments.
Unlike planar graphs, we may not change the position of the vertices --- these are fixed at the coordinates of the points of $\pts$.
See Figure \ref{fi:PlaneGraph}.
Plane graphs have a wide variety of applications (for many of those, see \cite{TORG17}).

\begin{figure}[ht]
\centerline{\includegraphics[width=0.4\textwidth]{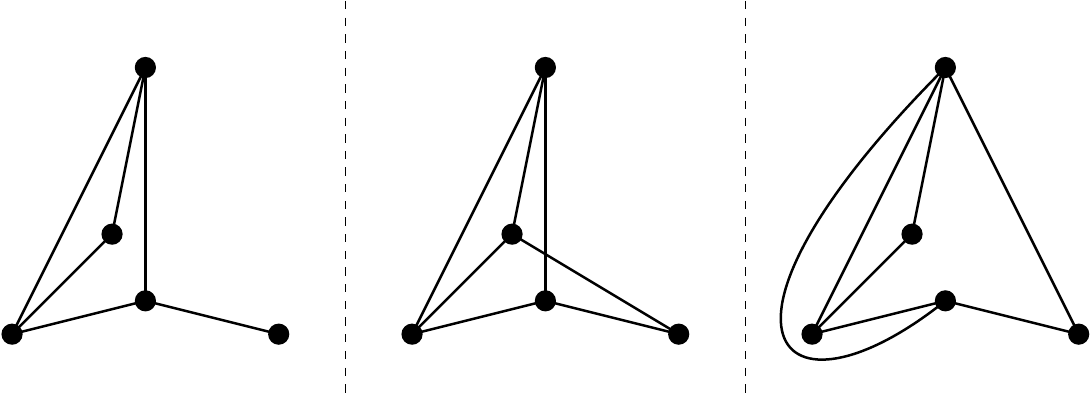}}
\caption{Three drawings of graphs on the same point set. All three are planar, but only the one on the left is a plane graph.}
\label{fi:PlaneGraph}
\end{figure}

We cannot use standard random graph models where each edge is chosen with a fixed probability.
If we do that, the resulting graph will contain edge crossings with a positive probability.
Instead, a random plane graph of a finite set $\pts$ is uniformly chosen from the set of all plane graphs that can be drawn on $\pts$.

For example, let $\pts$ be the set of three non-collinear points.
There are three potential edges and no two intersect, so $2^3=8$ plane graphs can be drawn on $\pts$.
We only consider the case where the vertices are labeled, so isomorphic graphs are distinct plane graphs of $\pts$.

Currently, the study of random plane graphs is focused on the expected degrees in a random graph. 
In a standard graph, this is a simple expectation calculation.
For plane graphs, this is a difficult problem with interesting applications.
For example, Lemma \ref{le:NumberOfPG} below shows a connection between expected degrees and the minimum number of plane graphs that can be drawn on any $n$ points. 

When studying properties of plane graphs, it is common to assume that no three points are collinear. 
This avoids a degenerate behavior that is reasonable to ignore.
Throughout this paper, we assume that no three points are collinear in all point sets.

For a finite set $\pts$, let $\hv_i(\pts)$ be the expected number of degree $i$ vertices in a random graph of $\pts$.
For a more rigorous definition and a discussion, see Section \ref{sec:theorems}.
The following theorem summarizes the current best results for expected degrees in random plane graphs (see \cite{SS13}).

\begin{theorem} \label{th:Previous}
Every set $\pts$ of $n$ points satisfies
\[ \hv_0(\pts)> \frac{n}{3207}, \qquad \hv_1(\pts)\ge \frac{3n}{1024}, \qquad \hv_2(\pts)\ge \frac{33n}{2048}. \]
We also have that
\[ \hv_2(\pts)+\hv_3(\pts)\ge \frac{n}{24}. \]
\end{theorem}

In the current work, we complement the bounds of Theorem \ref{th:Previous}, by deriving upper bounds on the expected degrees in a random plane graph.
Motivated by an application described below, we are most interested in the expected number of isolated vertices in a random plane graph.

\begin{theorem} \label{th:v0}
Let $\pts$ be a set of $n\ge 5$ points with a triangular convex hull. 
Then, 
\[ \hv_0(\pts) < \frac{11n}{112} < \frac{n}{10.18}. \] 
\end{theorem}

In Lemma \ref{le:construction}, we construct a point set with a triangular convex hull where $\hv_0(\pts)>\frac{n}{23.32}$.
We suspect that the correct bound is much closer to $n/23.32$ than to $n/10.18$.

For the case of vertices with a degree larger than 0, we derive the following result.

\begin{theorem} \label{th:AllDegs}
Let $\pts$ be a set of $n$ points.
Then, for every $i\ge 1$, we have that
\[ \hv_i(\pts) < \frac{n}{\sqrt{\pi i}}. \]
\end{theorem} 

We now mention the application that motivated us to derive an upper bound on $\hv_0(\pts)$. 
It is known the every set of $n$ points has $\Omega(11.65^n)$ plane graphs that can be drawn on it \cite{AHHHKV07}.
This bound is tight, since $n$ points in convex position have this number of plane graphs. 
In contrast, it seems plausible that, when the number of points on the convex hull of the point set is small, the number of plane graphs is significantly larger.
When the number of points of the convex hull does not depend on $n$, we are only aware of point sets with $\Omega(23.31^n)$ plane graphs.

Let $\pg^\triangle(n)$ denote the minimum number of plane graphs that $n$ points with a triangular convex hull may have. 
Lemma \ref{le:construction} (from Section \ref{sec:Additional}) states that $\pg^\triangle(n)=O(23.32^n)$. 
We also prove the following lemma in Section \ref{sec:Additional}.
It is based on similar lemmas from \cite{SS11,SS13,SW06}.
For a point set $\pts$ with a triangular convex hull, let $\hv_0^\triangle(\pts)$ be the expected number of 0-vings that are not vertices of the convex hull of $\pts$.

\begin{lemma} \label{le:NumberOfPG}
Consider a fixed $n \geq 2$. 
Let $\delta_{n}>0$ satisfy
$\hv_0^\triangle(\pts) \le \delta_n n$ for every set $\pts$ of $n$
points in $\RR^2$. Then
\[\pg^\triangle(n) \ge \frac{1}{\delta_n} \pg^\triangle(n-1). \]
\end{lemma}

Combining Theorem \ref{th:v0} with Lemma \ref{le:NumberOfPG} leads to a minimum of $\Omega(10.18^{n})$ plane graphs,\footnote{Theorem \ref{th:v0} provides an upper bound for $\hv_0(\pts)$ rather than $\hv_0^\triangle(\pts)$. However, since the convex hull is a triangle, we have that $\hv_0^\triangle(\pts)\le\hv_0(\pts)-3$. 
When $n$ is sufficiently large, $\hv_0^\triangle(\pts)<11n/112-3< n/10.18$.} which is not a new result.
However, further improving $n/10.18$ to $n/c$ for $c>11.65$ will lead to a new result. 
It seems plausible that our approach could be pushed further to obtain such a result. 
We suggest this as a direction for future work. 

\emph{charging schemes} are a common technique in graph theory, where one moves charge among vertices of a graph.
For example, the proof of the four color theorem involves a charging scheme (for example, see \cite{AH77,RSST97}).
Our proofs rely on a \emph{cross-graph charging scheme}, where charge moves between vertices of different graph of the same point set.
Such interaction between different graphs allows us to study properties of the average graph, or the expected properties of a random graph.

In Section \ref{sec:theorems}, we provide the main portions of the proofs of Theorems \ref{th:v0} and \ref{th:AllDegs}.
In Section \ref{sec:Lemmas}, we prove technical lemmas that are used in Section \ref{sec:theorems}.
Finally, in Section \ref{sec:Additional}, we prove Lemma \ref{le:NumberOfPG} and study a construction with $\hv_0(\pts)>\frac{n}{23.32}$.

\section{Proofs of Theorems \ref{th:v0} and \ref{th:AllDegs}} \label{sec:theorems}

In this section, we prove Theorems \ref{th:v0} and \ref{th:AllDegs}. 
To focus on the main aspects of these proofs, we postpone the proofs of some lemmas to Section \ref{sec:Lemmas}.
We begin with the proof of Theorem \ref{th:AllDegs}, since the proof of Theorem \ref{th:v0} begins in a similar manner and then has several additional steps.

We denote the set of plane graphs of $\pts$ as $\pgset(\pts)$ and set $\pg(\pts) = |\pgset(\pts)|$.
For a graph $G\in \pgset(\pts)$, we denote the number of degree $i$ vertices in $G$ as $v_i(G)$.
We denote the expected value of $v_i(G)$, for a graph $G$ uniformly chosen from $\pgset(\pts)$, as
\[ \hv_i(\pts) = \frac{
\sum_{G\in \pgset(\pts)} v_i(G)}{\pg(\pts)}. \]

A \emph{ving} (short for Vertex IN Graph) is a pair $(p,G)\in \pts \times \pgset(\pts)$.
Intuitively, rather than a point of $\pts$, a ving is an instance of that point in a specific plane graph.
The degree of the ving $(p,G)$ is the degree of $p$ in $G$. 
We also refer to a ving of degree $i$ as an $i$-ving.
See Figure \ref{fi:Vings}.
Our proofs rely on charging schemes where charge moves among vings from different graphs (rather than the more common transfer of charge between vertices of the same graph).

\begin{figure}[ht]
\centerline{\includegraphics[width=0.4\textwidth]{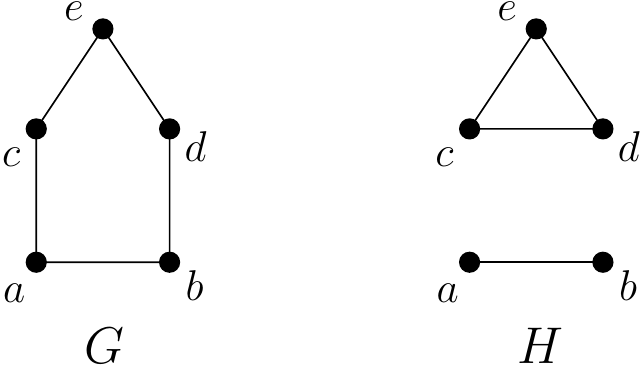}}
\caption{Plane graphs $G,H$ of the same point set $\pts=\{a,b,c,d,e\}$. Here, $(a,G)$ is a 2-ving, while $(a,H)$ is a 1-ving.}
\label{fi:Vings}
\end{figure}

Consider a graph $G$ and two vertices $u,v$ from $G$.
We say that $v$ \emph{is visible from} $u$ if the edge $uv$ is not in $G$ and does not intersect any edge of $G$.
For example, in Figure \ref{fi:Vings}, the vertices that are visible from $a$ in $G$ are $d,e$. 
The vertices that are visible from $a$ in $H$ are $c,d$.
The \emph{visibility} of a ving $(p,G)$ is the number of vertices that are visible from $p$ in $G$.
In Figure \ref{fi:Vings}, the ving $(a,G)$ has visibility 2, and $(e,H)$ has visibility 0.

We are now ready to prove Theorem \ref{th:AllDegs}. 
We first recall the statement of this result.
\vspace{2mm}

\noindent {\bf Theorem \ref{th:AllDegs}.}
\emph{Let $\pts$ be a set of $n$ points.
Then, for every $i\ge 1$, we have that}
\[ \hv_i(\pts) < \frac{n}{\sqrt{\pi i}}. \]
\begin{proof}
We fix an $i\ge 1$ and give every $i$-ving of every graph $G\in \pgset(\pts)$ a charge of 1.
All vings with degrees other than $i$ start with 0 charge. 
Let $t$ be the total charge summed over all vings of all graphs of $\pgset(\pts)$. 
Then 
\begin{equation} \label{eq:ChargeInitial}
t=\sum_{G\in \pgset(\pts)} v_i(G) = \hv_i(\pts) \cdot \pg(\pts).
\end{equation}
We will redistribute the charge among vings of all degrees and then derive an upper bound on the charge that a ving may have.
This will lead to an upper bound on $t$. 
Combining this bound with \eqref{eq:ChargeInitial} will complete the proof.

\begin{figure}[ht]
\centerline{\includegraphics[width=0.6\textwidth]{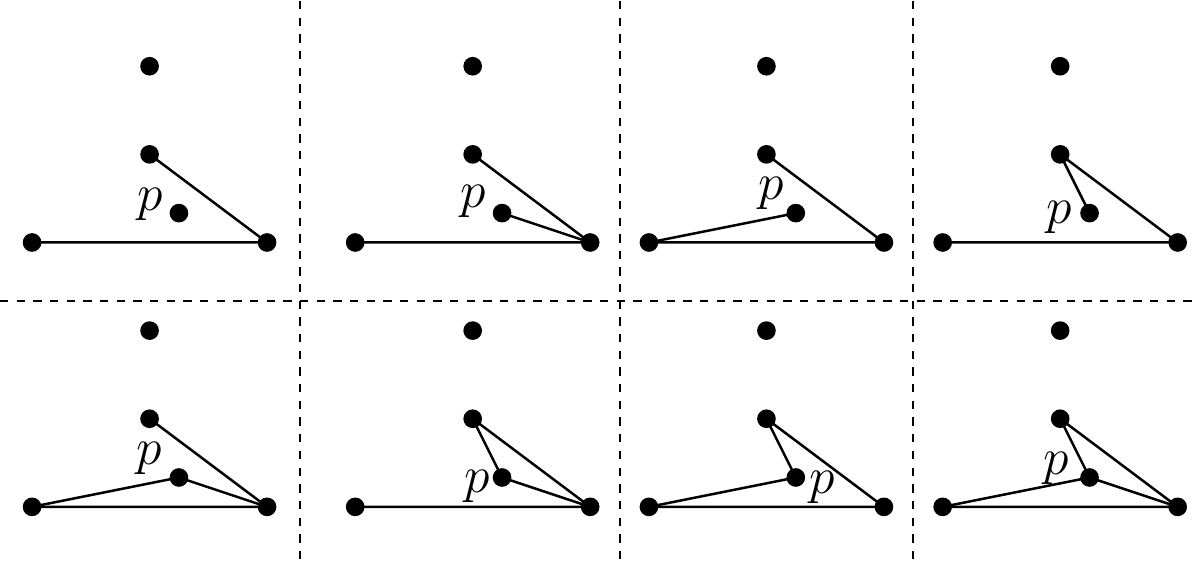}}
\caption{The family of the 0-ving $(p,G)$ in the top left corner. This is a 3-family, so it has eight vings.}
\label{fi:family}
\end{figure}

The \emph{family} of a 0-ving $v=(p,G)\in \pts \times \pgset(\pts)$ is the set of all vings obtained by starting with $v$ and connecting $p$ to any number of vertices visible from $p$.
For example, see Figure \ref{fi:family}.
If $p$ has visibility $j$ in $G$, then we refer to the corresponding family as a $j$\emph{-family}.
We note that a $j$-family consists of $2^j$ vings, all with the same point of $\pts$.
We also say that \emph{the visibility of the family} is $j$.
We claim that each ving $(q,H)$ belongs to exactly one family. 
Indeed, $(q,H)$ belongs to the family of the 0-ving obtained by removing from $H$ all edges that have $q$ as an endpoint.

Every $i$-ving belongs to a family of visibility at least $i$.
For $j\ge i$, the number of $i$-vings in a $j$-family is $\binom{j}{i}$.
Thus, the total charge in a $j$-family is $\binom{j}{i}$.
We redistribute this charge evenly among all $2^j$ vings of the family. 
Then, each ving of the $j$-family has a charge of
\[ \frac{\binom{j}{i}}{2^j} = i! \cdot \frac{j\cdot (j-1)\cdots (j-i+1)}{2^j}. \] 

We study the maximum charge that a ving can have after the above redistribution. 
That is, we study the maximum of the above expression, as a function of $j$.
When looking for this maximum, we may ignore the factor $i!$, since it does not depend on $j$.
We thus define 
\[ f(x) = \frac{x\cdot (x-1) \cdots(x-i+1)}{2^x}.\]

We are only interested in the range $x\ge i$.
In this range, we wish to find when $f(x)$ is increasing and when it is decreasing.
The quotient rule states that
\[ f'(x) = \frac{X \cdot 2^x - 2^x 
\cdot \ln 2 \cdot x\cdot (x-1) \cdots(x-i+1)}{2^{2x}}, \]
where
\[ X = \frac{x\cdot (x-1) \cdots(x-i+1)}{x} + \frac{x\cdot (x-1) \cdots(x-i+1)}{x-1} + \cdots + \frac{x\cdot (x-1) \cdots(x-i+1)}{x-i+1}. \]
Simplifying the above leads to 
\[ f'(x) = \frac{x\cdot (x-1) \cdots(x-i+1)}{2^x}\cdot \left(\frac{1}{x}+ \frac{1}{x-1}+\cdots+\frac{1}{x-i+1} - \ln 2\right).\]

In our range of $x\ge i$, we have that $\frac{x\cdot (x-1) \cdots(x-i+1)}{2^x}> 0$.
That is, the sign of $f'(x)$ depends only on the expression in the parentheses to the right.
This expression is negative for any sufficiently large $x$, and positive for small $x\ge i$. 
To find the maximum value of $f(x)$, we search for the $x$ where the sign of $f'(x)$ changes.

As shown in \cite{BW71}, the Euler--Maclaurin formula implies that 
\begin{equation} \label{eq:Harmonic}
\sum_{k=1}^m \frac{1}{k}= \ln m + \gamma +\frac{1}{2m} - \eps_m, 
\end{equation}
where $\gamma \approx 0.5772$ (the Euler--Mascheroni constant) and $0\le \eps_n \le 1/8m^2$.

By applying \eqref{eq:Harmonic} twice, with $x=2i$ and $x=i$, we obtain
\begin{align*} 
\frac{1}{2i}+ \frac{1}{2i-1} &+\cdots+\frac{1}{i+1} =  \sum_{k=1}^{2i} \frac{1}{k} - \sum_{k=1}^{i} \frac{1}{k} 
\\
&= \ln 2i + \gamma +\frac{1}{4i} - \eps_{2i} - (\ln i + \gamma +\frac{1}{2i} - \eps_i) = \ln 2 - \frac{1}{4i} - \eps_{2i}+\eps_i < \ln 2.
\end{align*}
For the final transition, we note that $\eps_i\le 1/8i^2 < 1/4i$.
This implies that $f'(2i)<0$.

We next compute
\[ f(2i) = \frac{(2i)\cdot(2i-1)\cdots(i+1)}{2^{2i}} = \frac{(2i-1)\cdot(2i-1)\cdots (i)}{2^{2i-1}} = f(2i-1). \]
Combining this with $f'(2i)<0$ implies that $f'(x)$ switches from positive to negative for some $2i-1<x<2i$.
This in turn implies that the maximum charge of a ving is obtained for both $j=2i$ and $j=2n-1$. 

One variant of Stirling's approximation \cite{Robbins55} states that 
\[ \sqrt{2\pi m}\left(m/e\right)^m e^{\frac{1}{12m+1}} < m! < \sqrt{2\pi m}\left(m/e\right)^m e^{
\frac{1}{12m}}.\]
Thus, the maximum ving charge is 
\[ \frac{\binom{2i}{i}}{2^{2i}} = \frac{(2i)!}{(i!)^2\cdot 2^{2i}} < \frac{\sqrt{4\pi i}\left(2i/e\right)^{2i} e^{\frac{1}{24i}}}{2\pi i\left(i/e\right)^{2i} e^{\frac{2}{12i+1}}\cdot 2^{2i}} = \frac{\sqrt{\pi i} e^{\frac{1}{24i}}}{\pi i e^{\frac{2}{12i+1}}} < \frac{1}{\sqrt{\pi i}}. \] 

The above implies that 
\[ t < \sum_{G\in \pgset(\pts)} \frac{n}{\sqrt{\pi i}} = \pg(\pts) \cdot \frac{n}{\sqrt{\pi i}}. \]
Combining this with \eqref{eq:ChargeInitial} leads to $\hv_i(\pts) < n/\sqrt{\pi i}$.
\end{proof}

Let $\pts$ be a set of points in $\RR^2$. 
We say that a point $p\in \pts$ is a \emph{hull point} if $p$ is a vertex of the convex hull of $\pts$.
If $p$ is not a hull point, then we say that it is an \emph{internal point}.
Similarly, a vertex in a graph of $\pgset(\pts)$ is a \emph{hull vertex} if it corresponds to a hull point of $\pts$.
Otherwise, it is an \emph{internal vertex}.

\begin{figure}[ht]
\centerline{\includegraphics[width=0.45\textwidth]{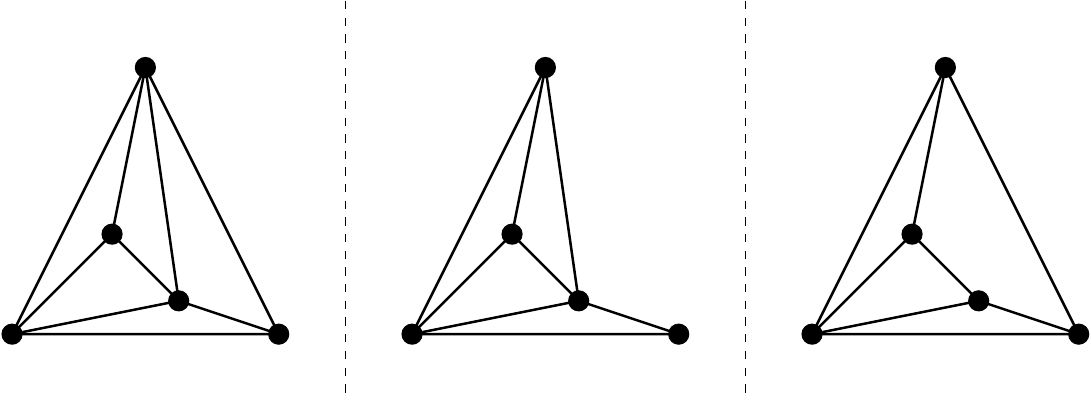}}
\caption{Three plane graphs of the same point set. Only the leftmost one is a triangulation.}
\label{fi:triangulation}
\end{figure}

A \emph{triangulation} of $\pts$ is a plane graph $G\in \pgset(\pts)$ where all the bounded faces are triangles and all edges of the convex hull are present. 
See Figure \ref{fi:triangulation}.
Equivalently, a triangulation is a maximal plane graph, in the sense that no additional edges can be added. 
Indeed, if a bounded face is not a triangle, we can always add one of its diagonals. 
In the other direction, if we can add an edge to a plane graph, then an edge of the convex hull is missing or one of the bounded faces is not a triangle.

We are now ready to prove Theorem \ref{th:v0}.
We first recall the statement of this result.
\vspace{2mm}

\noindent {\bf Theorem \ref{th:v0}.}
\emph{Let $\pts$ be a set of $n\ge 5$ points with a triangular convex hull. 
Then,}
\[ \hv_0(\pts) < \frac{11n}{112} < \frac{n}{10.18}. \] 

\begin{proof}
We begin as in the proof of Theorem \ref{th:AllDegs}.
That is, we give every 0-ving in every graph of $\pgset(\pts)$ a charge of 1.
All other vings start with a charge of 0.
Let $t$ be the total charge, summing up the charge of all vings in all graphs of $\pgset(\pts)$. 
Then 
\begin{equation} \label{eq:ChargeInitial2a}
t=\sum_{G\in \pgset(\pts)} v_0(G) = \hv_0(\pts) \cdot \pg(\pts).
\end{equation}

We consider families of vings, as defined in the proof of Theorem \ref{th:AllDegs}.
Since each family has exactly one 0-ving, the total charge in each family is 1. 
We equally distribute this charge among the vings of the family. 
Thus, each ving in a $j$-family receives a charge of $1/2^j$.

See Section \ref{sec:Lemmas} for a proof of the following lemma. 

\begin{lemma} \label{le:visibility}
Every 0-ving of every graph $G\in \pgset(\pts)$ has visibility at least 3. 
\end{lemma}

Lemma \ref{le:visibility} implies that the maximum ving charge is $1/2^3=1/8$.
This implies that 
\[ t\le \sum_{G\in \pgset(\pts)} \frac{n}{8} = \pg(\pts) \cdot \frac{n}{8}. \]
Combining this with \eqref{eq:ChargeInitial2a} leads to $\hv_0(\pts) \le n/8$.

\parag{Studying the charge in a graph $G$.}
To improve the above bound, we study the total charge that a graph $G\in \pgset(\pts)$ may contain. 

\begin{figure}[ht]
\centerline{\includegraphics[width=0.35\textwidth]{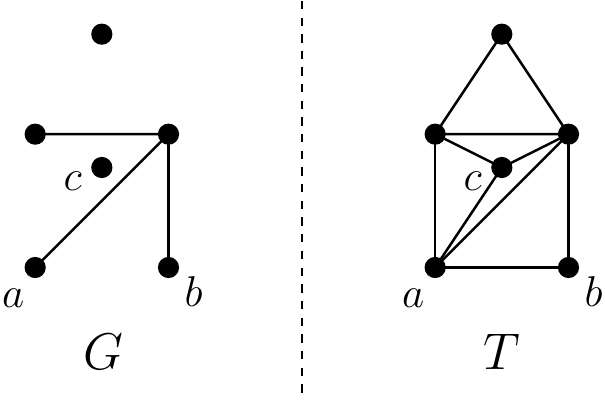}}
\caption{The triangulation $T$ contains the plane graph $G$. 
In $G$, the potentials of $a,b,c$ are $4,2,3$, respecitively. 
In $T$, the degrees of $a,b,c$ are $4,2,3$, respectively. }
\label{fi:potential}
\end{figure}

We fix a $G\in \pgset(\pts)$ and a ving $v=(p,G)$.
Let the \emph{potential} of $v$ be the degree of $p$ in $G$ plus the number of vertices of $G$ that are visible from $G$.
In other words, the potential of $v$ is the number of vertices it is connected to or can be connected to.
See the left side of Figure \ref{fi:potential}.
We denote this quantity as $\pot(p,G)$.

The potential of a ving $v=(p,G)$ is equal to the visibility of the family of $v$.
Indeed, it is the visibility of the 0-ving obtained by removing all edges adjacent to $p$ in $G$. 
Thus, the total charge in $G$ is
\[ \sum_{p\in \pts} \frac{1}{2^{\pot(p,G)}}. \]

By definition, $G$ is contained in a triangulation $T\in \pgset(\pts)$.
Indeed, we can obtain such a triangulation by repeatedly adding edges to $G$, until this is no longer possible.
If $G$ is contained in multiple triangulations, we arbitrarily choose one to be $T$.

We claim that every $p\in \pts$ satisfies that $\pot(p,G)\ge \pot(p,T)$.
Indeed, since a triangulation is a maximal plane graph, no vertex is visible from another vertex in $T$.
That is, $\pot(p,T)$ is the degree of $p$ in $T$.
Every vertex that is connected to $p$ in $T$ is either also connected to $p$ in $G$ or visible from $p$ in $G$.
See Figure \ref{fi:potential}.
See Section \ref{sec:Lemmas} for a proof of the following lemma. 

\begin{lemma} \label{le:deg3}
Every triangulation $T\in \pgset(\pts)$ satisfies that $v_3(T)\le 2n/3-1$.
\end{lemma}

By Lemma \ref{le:deg3}, at most $2n/3-1$ vings of $G$ have a charge of $1/8$. 
Every other ving has a charge of at most $1/16$.
This implies that 
\[ t\le \sum_{G\in \pgset(\pts)}\left((2n/3-1)\cdot\frac{1}{8}+(n/3+1)\cdot\frac{1}{16} \right)  = \pg(\pts) \cdot \frac{5n-3}{48}. \]
Combining this with \eqref{eq:ChargeInitial2a} leads to $\hv_0(\pts) < \frac{5n}{48} = \frac{n}{ 9.6}$.

\parag{A more careful study of degrees in a triangulation.}
By inspecting the proof of Lemma \ref{le:deg3} in Section \ref{sec:Lemmas}, we observe that when a triangulation has $2n/3-1$ degree 3 vertices, every other internal vertex has degree at least 6. 
That is, a triangulation cannot have $2n/3-1$ vertices of degree 3 and $n/3+1$ vertices of degree 4.
The following lemma is based on this observation. 
See Section \ref{sec:Lemmas} for a proof. 

\begin{lemma} \label{le:deg34}
Every triangulation $T\in \pgset(\pts)$ satisfies that
\[ v_4(T)\le \frac{6n-9v_3(T)-6}{2}. \]
\end{lemma}

Let $v_\text{large}(T) = n- v_3(T) - v_4(T)$.
Then the total charge in $T$ is at most
\begin{equation} \label{eq:deg345}
\frac{v_3(T)}{8} + \frac{v_4(T)}{16} + \frac{v_\text{large}(T)}{32}. 
\end{equation}
We look for the largest value \eqref{eq:deg345} may have under the constraints of Lemmas \ref{le:deg3} and \ref{le:deg34}.
We also include the additional constraint $v_3(T)+v_4(T)+v_\text{large}(T)=n$.

As a starting point, when $v_3(T)=2n/3-1$, Lemma \ref{le:deg34} implies that $v_4(T)\le 1$.
In this case, the value of \eqref{eq:deg345} is 
\[ \frac{2n/3-1}{8}+\frac{1}{16}+\frac{n/3}{32} = \frac{3n-2}{32}. \]

When decreasing $v_3(T)$ by 1, we may increase $v_4(T)$ by at least 4, which increases the total charge.
Indeed, 
\begin{itemize}[noitemsep,topsep=1pt]
\item The removed 3-ving decreases the charge by $1/8$.
\item At least four new 4-vings increase the charge by at least $4/16 = 1/4$. 
\item We lose three larger vings, which decreases the charge by $3/32$.
\end{itemize}

To increase \eqref{eq:deg345}, we repeatedly decrease the number of 3-vings and increase the number of 4-vings.
We can keep doing that until $v_3(T)=(4n-6)/7$.
Then, the bound of Lemma \ref{le:deg34} implies that $v_3(T)+v_4(T)=n$, so $v_\text{large}=0$.
Thus, a further decrease of $v_3(T)$ by 1 would only increase $v_4(T)$ by 1, which decreases the total charge. 
We conclude that the maximum value of  \eqref{eq:deg345} is 
\[ \frac{(4n-6)/7}{8}+\frac{(3n+6)/7}{16}+\frac{0}{32} = \frac{11n-6}{112}. \]

As a sanity check, we also ran this optimization problem in Wolfram Mathematica \cite{Wolfram}, obtaining the same result. 
We conclude that
\[ t\le \sum_{G\in \pgset(\pts)}\left(\frac{11n-6}{112} \right)  = \pg(\pts) \cdot \frac{11n-6}{112}. \]
Combining this with \eqref{eq:ChargeInitial2a} leads to $\hv_0(\pts) < \frac{11n}{112}$.
\end{proof}

It seems plausible that a more careful analysis would lead to a better bound on the charge that a triangulation can have, thus improving the bound of Theorem \ref{th:v0}. 
Further progress may be obtained by considering the charge in an \emph{average triangulation}. 
We leave these approaches as open problems for future work.

\section{Proofs of Lemmas} \label{sec:Lemmas}

In this section, we prove the lemmas that were stated in Section \ref{sec:theorems}.
All lemmas study properties of plane graphs.
\vspace{2mm}

\begin{figure}[ht]
\centerline{\includegraphics[width=0.4\textwidth]{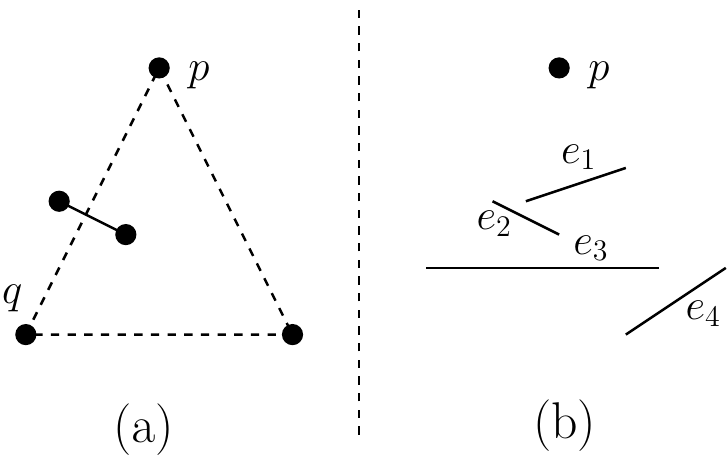}}
\caption{(a) If $q$ is not visible from $p$, then there exists a vertex outside of the triangle. (b) The only valid order from $p$ is $e_1 < e_2 < e_3 < e_4$. Since $e_1$ is smallest, both its endpoints are visible from $p$. }
\label{fi:visibility3}
\end{figure}

\noindent {\bf Lemma \ref{le:visibility}.}
\emph{Every 0-ving of every graph $G\in \pgset(\pts)$ has visibility at least 3. 
}

\begin{proof}
We fix a graph $G\in \pgset(\pts)$ and a ving $v=(p,G)$.
We first consider the case where $p$ is a hull point.
Since the convex hull of $\pts$ is a triangle, the other two vertices of this triangle are visible from $p$ in $G$. 
Indeed, if an edge $e$ blocks the visibility of a hull vertex from $p$, then one vertex of $e$ is outside of the triangle. 
See Figure \ref{fi:visibility3}(a).
Since this is impossible, the visibility of $(p,G)$ is at least 2.

By the assumption that $n\ge 5$, the triangular convex hull is not empty. 
We may assume that there are edges inside of the convex hull, since otherwise we are done. 
We define a partial order on these edges, as follows.
We say that an edge $e$ is smaller than an edge $e'$ if we can shoot a ray from $p$ that first hits $e$ and then hits $e'$.
See Figure \ref{fi:visibility3}(b).
Since the edges do not intersect, this partial order is well-defined.

We say that an edge $e$ is \emph{smallest}
if no other edge is smaller than $e$ with respect to the partial order above. 
Since this order is partial, there may be more than one smallest edge.
We arbitrarily consider a smallest edge $e$.
By definition, both endpoints of $e$ are visible from $p$.
See Figure \ref{fi:visibility3}(b).
Since at most one of these endpoints is a hull vertex, the visibility of $p$ is at least 3.

\begin{figure}[ht]
\centerline{\includegraphics[width=0.4\textwidth]{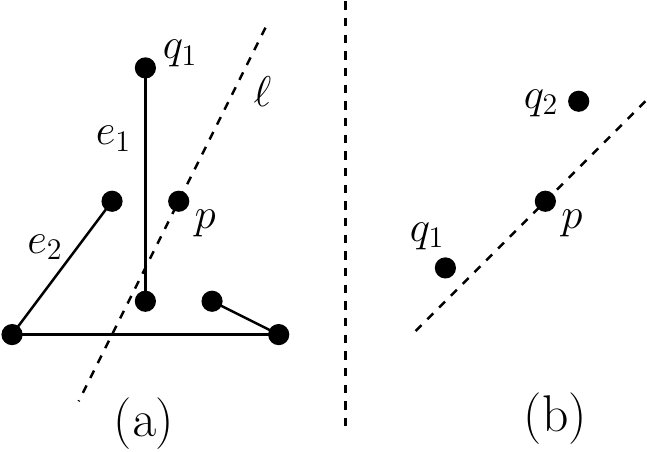}}
\caption{(a) In the half-plane above $\ell$, the edge $e_1$ is a smallest edge. At least one endpoint of such an edge is in the half-plane and visible from $p$. (b) There always exists a line that contains $p$ and has $q_1,q_2$ on the same side. }
\label{fi:halfPlane}
\end{figure}

\parag{The case where $p$ is an internal vertex.}
We next consider the case where $p$ is an internal vertex.
Let $\ell$ be a line that contains $p$ but no other point of $\pts$.
We fix one open half-plane that is defined by $\ell$ and consider all edges of $G$ that are contained in this half-plane or intersect it. 
See Figure \ref{fi:halfPlane}(a).
We define a partial order on the edges, as above.
Let $e$ be a smallest edge with respect to this order.
At least one endpoint $q_1$ of $e$ is in the half-plane.
By definition, $q_1$ is visible from $p$.

We repeat the above with the other open half-plane defined by $\ell$, to obtain a point $q_2$ visible from $p$.
Since no three points of $\pts$ are collinear, there exists a line $\ell'$ such that $\ell' \cap \pts = \{p\}$ and $q_1,q_2$ are on the same side of $\ell'$.
See Figure \ref{fi:halfPlane}(b).
By repeating the above argument with $\ell'$ rather than $\ell$, we get a point $q_3$ that is visible from $p$.
Thus, the visibility of $p$ is at least 3.
\end{proof}

\noindent {\bf Lemma \ref{le:deg3}.}
\emph{Every triangulation $T\in \pgset(\pts)$ satisfies $v_3(T)\le 2n/3-1$.}
\vspace{-2mm}

\begin{proof}
We fix a triangulation $T\in \pgset(\pts)$.
By the assumption $n\ge 5$, the convex hull of $\pts$ is not empty. 
We claim that at most one of the three vertices of the convex hull is of degree 3. 
Indeed, consider a 3-ving $(p,T)$ that is a hull vertex. 
Let $q,r$ be the other hull points and let $a$ be the internal point that is connected to $p$ in $T$.
See Figure \ref{fi:Deg3}.
Since $(p,T)$ is a 3-ving, the edges $pq$ and $pr$ exist in $T$ and the triangles $\triangle pqa, \triangle pra$ are empty.

\begin{figure}[ht]
\centerline{\includegraphics[width=0.6\textwidth]{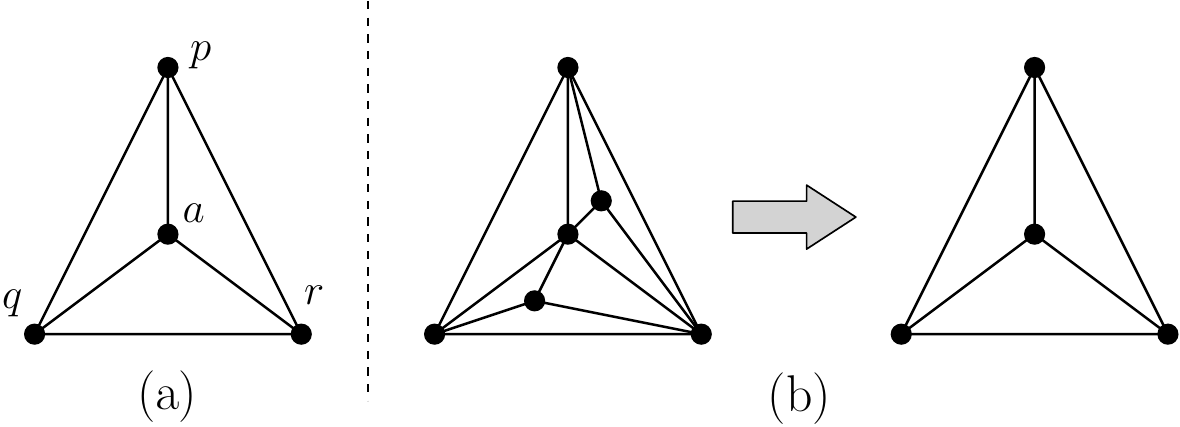}}
\caption{(a) The setup at the beginning of the proof of Lemma \ref{le:deg3}. (b) Removing internal 3-vings leads to another triangulation. }
\label{fi:Deg3}
\end{figure}

Since $n\ge 5$, there is at least one more vertex in $\triangle qra$. 
This implies that the degrees of $(q,T)$ and $(r,T)$ are at least 4. 
That is, at most one hull vertex has degree 3. 

\parag{Studying internal vertices.}
Let $\pts_T$ be the set of points of $\pts$ that have degree larger than 3 in $T$ or are part of the convex hull (or both).
By removing from $T$ every internal vertex of degree 3, we obtain a triangulation $T'$ of $\pts_T$.
See Figure \ref{fi:Deg3}(b).

A triangulation of $m$ points with a triangular convex hull has $2m-4$ faces (for example, see \cite[Chapter 12]{Bona02}).
Thus, a triangulation of $\pts_T$ has $2|\pts_T|-5$ triangles (not counting the unbounded face).
In a triangulation, two internal degree 3 vertices cannot be connected by an edge, so each triangle contains at most one of the removed vertices. 

Since at most one hull vertex has degree 3, we get that $|\pts_T|\le n - v_3(T) +1$.
Combining the above leads to
\[ v_3(T) \le 2|\pts_T|-5 \le 2(n-v_3(T)+1)-5 = 2n-2v_3(T) -3.  \]
Rearranging the above leads to $v_3\le 2n/3-1$.
\end{proof}

\noindent {\bf Lemma \ref{le:deg34}.}
\emph{Every triangulation $T\in \pgset(\pts)$ satisfies that}
\[ v_4(T)\le \frac{6n-9v_3(T)-6}{2}. \]
\vspace{-2mm}

\begin{proof}
We begin in the same way as in the proof of Lemma \ref{le:deg3}.
We fix a triangulation $T\in \pgset(\pts)$, define the set $\pts_T$ as before, and let $T'$ be the triangulation obtained by removing from $T$ the vertices of $\pts\setminus\pts_T$.
That is, $T'$ is the triangulation of $\pts_T$ that is contained in $T$.

\begin{figure}[ht]
\centerline{\includegraphics[width=0.4\textwidth]{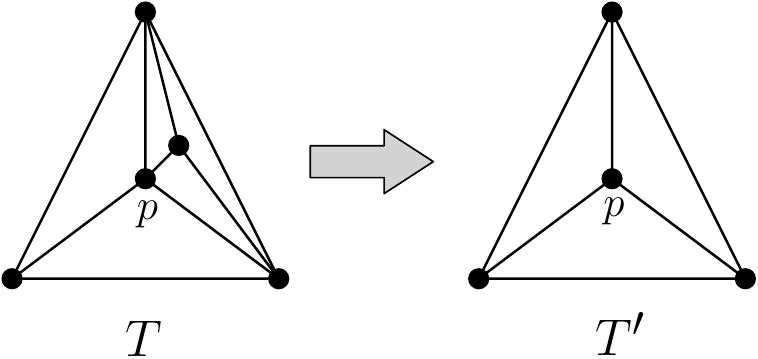}}
\caption{If $(p,T')$ is a 3-ving and $(p,T)$ is a 4-ving, then exactly one of the triangles surrounding $(p,T')$ is not empty.}
\label{fi:empty}
\end{figure}

We say that a triangle of $T'$ is \emph{empty} if it does not contain a vertex of $\pts\setminus \pts'$ that has degree 3 in $T$.
Consider a ving $(p,T)$ of degree 4 with internal $p$.
Then $(p,T')$ is of degree 3 or 4.
If $(p,T')$ is of degree 3, then exactly one triangle adjacent to it is not empty.  
If $(p,T')$ is of degree 4, then all four triangles next to it are empty. 
See Figure \ref{fi:empty}.

We recall that $T'$ contains $2|\pts_T|-5$ triangles.
Each empty triangle is adjacent to three points of $\pts_T$ and each degree 4 ving in $T'$ is adjacent to at least two empty triangles. 
Let $t_\text{empty}$ be the number of empty triangles in $T'$.
Since at most one hull vertex has degree 3, we conclude that
\begin{align*} v_4(T) &\le \frac{3\cdot t_\text{empty}}{2} \le \frac{3\cdot((2|\pts_T|-5)-(v_3(T)-1))}{2} \\[2mm]
&\le \frac{3\cdot(2(n-(v_3(T)-1))-5-v_3(T)+1)}{2} =  \frac{6n-9v_3(T)-6}{2}. 
\end{align*}
\end{proof}

\section{Additional proofs} \label{sec:Additional}

This section consists of two remaining proofs: A construction with $\hv_0(\pts)>\frac{n}{23.32}$ and Lemma \ref{le:NumberOfPG}.

\begin{figure}[ht]
\centerline{\includegraphics[width=0.12\textwidth]{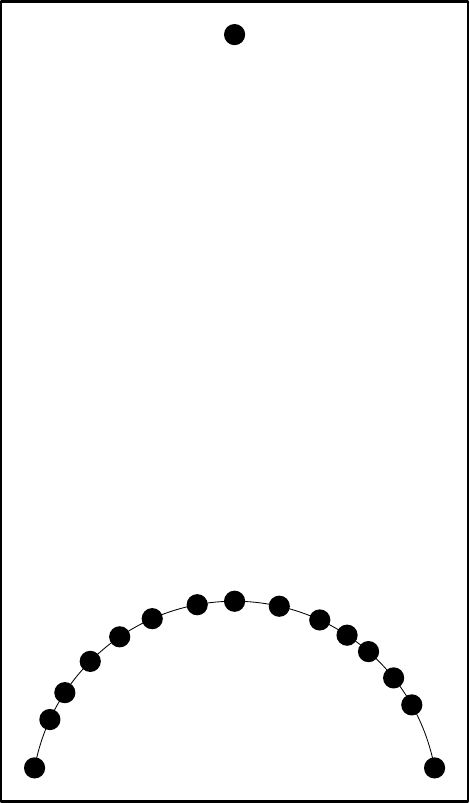}}
\caption{The construction of Lemma \ref{le:construction}: $n-1$ points on an upper half circle and one point high above.}
\label{fi:construction}
\end{figure}

\begin{lemma} \label{le:construction}
There exists a set $\pts$ of $n$ points with a triangular convex hull, approximately $23.31^n$ plane graphs, and $\hv_0(\pts)>\frac{n}{23.32}$.
\end{lemma}
\begin{proof}
Let $\pts$ be a set of $n$ points, as depicted in Figure \ref{fi:construction}.
In more detail, $n-1$ points $p_1,\ldots,p_{n-1}$ are on the top half of a circle.
The last point $q$ is above this half circle, sufficiently far away so that no edge of the form $qp_i$ intersects an edge of the form $p_jp_{j+1}$.

By \cite[Theorem 4]{FN99}, the number of plane graphs that can be drawn over $m$ points in convex position is
\begin{equation} \label{eq:GraphsForConv} 
\frac{1}{4}\sqrt{-120+99\sqrt{2}} \cdot \frac{(6+4\sqrt{2})^m}{\sqrt{\pi}m^{3/2}}\cdot \left(1+O\left(\frac{1}{m}\right)\right). 
\end{equation}

In our case, we can set $m=n-1$, to obtain the number of plane graphs drawn on $p_1,\ldots,p_{n-1}$.
This ignores the $n-1$ potential edges of the form $qp_i$.
Since these $n-1$ edges do not form any intersections, there are $2^{n-1}$ choices for which edges are in the graph.
We conclude that 
\[ \pg(\pts) = \Theta(2^{n-1} \cdot (6+4\sqrt{2})^{n-1})\approx \Theta(23.31^n).\]

We fix $p_i$ for some $1 \le i \le n-1$.
There is a natural bijection between 0-vings of $p_i$ and plane graphs of $\pts\setminus\{p_i\}$.
This bijection removes a 0-ving of $p_i$ from a plane graph of $\pts$. 
In the other direction, it adds $p_i$ to a plane graph of $\pts\setminus\{p_i\}$, as a 0-ving.
Note that $\pts\setminus\{p_i\}$ is the same point configuration with $n-1$ points rather than $n$.

We set $c=\frac{1}{4\sqrt{\pi}}\sqrt{-120+99\sqrt{2}}$. 
By applying \eqref{eq:GraphsForConv} twice, the portion of graphs of $\pgset(\pts)$ where $p_i$ is a 0-ving is
\[ \frac{\pg(\pts\setminus\{p_i\})}{\pg(\pts)} = \frac{2^{n-2}c \cdot \frac{(6+4\sqrt{2})^{n-2}}{{(n-2)}^{3/2}}\cdot \left(1+O\left(\frac{1}{n-2}\right)\right)}{2^{n-1}c \cdot \frac{(6+4\sqrt{2})^{n-1}}{{(n-1)}^{3/2}}\cdot \left(1+O\left(\frac{1}{n-1}\right)\right)} = \frac{ \frac{1}{{(n-2)}^{3/2}}\cdot \left(1+O\left(\frac{1}{n-2}\right)\right)}{2 \cdot \frac{6+4\sqrt{2}}{{(n-1)}^{3/2}}\cdot \left(1+O\left(\frac{1}{n-1}\right)\right)}. \]
When $n$ is sufficiently large, this expression is larger than $\frac{1}{23.314}$.

It remains to consider the top vertex $q$. 
This vertex forms a 0-ving when none of the edges $qp_i$ exist. 
By repeating the above calculation, we obtain that the portion of graphs of $\pgset(\pts)$ where $q$ is a 0-ving is $\frac{1}{2^{n-1}}$.
Combining the above leads to
\[ \hv_0(\pts) > \frac{n-1}{23.314} + \frac{1}{2^{n-1}}. \]
When $n$ is sufficiently large, we obtain the assertion of the lemma.
\end{proof}

For a plane graph $G$ of a point set with a triangular convex hull, let $v_0^\triangle(G)$ be the number of internal vertices of degree 0 in $G$.
\vspace{2mm}

\noindent {\bf Lemma \ref{le:NumberOfPG}.}
\emph{Consider a fixed $n \geq 2$. 
Let $\delta_{n}>0$ satisfy
$\hv_0^\triangle(\pts) \le \delta_n n$ for every set $\pts$ of $n$
points in $\RR^2$. Then}
\[\pg^\triangle(n) \ge \frac{1}{\delta_n} \pg^\triangle(n-1). \]

\begin{proof}
Let $\pts$ be a set that minimizes $\pg(\pts)$
among all sets of $n$ points.
We can get some plane graphs of $\pts$ by
choosing an internal point $q \in \pts$ and
a plane graph $H$ of $\pts \setminus \{q\}$, and inserting $q$ to $H$ with no edges.
A plane graph $G$ of $\pts$ can be obtained
in exactly $v_0^\triangle(G)$ ways in this manner.
This implies that
\[
\hv_0(\pts)\cdot \pg(\pts) = \sum_{G\in \pgset(\pts) } v_0^\triangle(G) = \sum_{q \in \pts} \pg(\pts \setminus\{q\}).
\]

The left part of the above equation equals $\hv_0^\triangle(\pts)\cdot \pg^\triangle(n)$.
The right part is at least $n\cdot \pg^\triangle(n-1)$. 
Recalling that $\hv_0^\triangle(\pts) \le \delta_n n$ leads to
\[ \pg^\triangle(n) = \pg(\pts) \ge \frac{n}{\hv_0^\triangle(\pts)} \cdot \pg(n-1) \ge
\frac{1}{\delta_n} \cdot \pg(n-1). \]
\end{proof}

\end{document}